\documentclass[11pt,a4paper,reqno]{amsart}

\usepackage[utf8]{inputenc}
\usepackage[english]{babel}
\usepackage{amsmath}

\usepackage{amssymb, mathabx}
\usepackage[numbers]{natbib}
\usepackage{graphicx}

\usepackage{braket}
\usepackage[colorlinks,hyperindex,bookmarksopen,linkcolor=red,citecolor=blue,urlcolor=blue]{hyperref}

\usepackage{mathrsfs}

\usepackage{braket}

\bibpunct{[}{]}{;}{n}{,}{,}
\usepackage[hmargin=2cm,vmargin={3cm,4cm}]{geometry}

\newtheorem{thm}{Theorem}[section]
\theoremstyle{definition}                                 %stile corsivo
\theoremstyle{definition}                           %stile roman
\theoremstyle{remark}                             %stile per osservazioni
\usepackage{color}

\usepackage{mathtools,slashed}
\newcommand{\be}{\begin{eqnarray}}
\newcommand{\ee}{\end{eqnarray}}

\numberwithin{equation}{section}
\allowdisplaybreaks

\usepackage{etoolbox}
\newcommand*{\affaddr}[1]{#1} % No op here. Customize it for different styles.
\newcommand*{\affmark}[1][*]{\textsuperscript{#1}}

\begin{document}
%\renewcommand{\theequation}{\thesection.\arabic{equation}}
%​\renewcommand{\thefigure}{\thesection.\arabic{figure}}
%​\renewcommand{\thetable}{\thesection.\arabic{table}}
\title[Free boundary problem for the role of planktonic cells in biofilms]{Free boundary problem for the role of planktonic cells \\ in biofilm formation and development}

\author{B.~D'Acunto\affmark[1]}

\author{L.~Frunzo\affmark[1]}

\author{V.~Luongo\affmark[1]}%\corref{cor1}\fnref{fn1}}

\author{M.R.~Mattei\affmark[1]}%\corref{cor1}\fnref{fn1}}

\author{A.~Tenore\affmark[1]}

\maketitle

\vspace{-0.5cm}

{\footnotesize
  \begin{center}
\affaddr{\affmark[1]University of Naples "Federico II", Department of Mathematics and Applications "Renato Caccioppoli", \\
via Cintia, Monte S. Angelo I-80126 Napoli, Italy}\\
%\affaddr{\affmark[2]Florida State University, Department of Mathematics, 208 Love Building, Tallahassee, FL 32306-4510, USA}\\
%\affaddr{\affmark[3]Texas A\&M University, Department of Civil \& Environmental Engineering, College Station, TX 77843, USA}\\
%\affaddr{\affmark[4]Texas A\&M University, Department of Chemical Engineering, College Station, TX 77843, USA}\\
   \end{center}}

	\begin{center}
	\footnotesize{Corresponding author: M.R.~Mattei, \texttt{mariarosaria.mattei@unina.it}}\\
   \end{center}
%\address[na1]{University of Naples "Federico II", Department of Mathematics and
%Applications, via Cintia, Monte S. Angelo
%I-80126 Napoli, Italy}
%
%\address[na2]{Florida State University, Department of Mathematics, 208 Love Building, Tallahassee, FL 32306-4510, USA}
%
%\address[na3]{Texas A\&M University, Department of Civil \& Environmental Engineering, College Station, TX 77843, USA}
%
%\address[na4]{Texas A\&M University, Department of Chemical Engineering, College Station, TX 77843, USA}

\begin{abstract}

The dynamics of biofilm lifecycle are deeply influenced by the surrounding environment and the interactions between sessile and planktonic phenotypes. Bacterial biofilms typically develop in three distinct stages: attachment of cells to a surface, growth of cells into colonies, and detachment of cells from the colony into the surrounding medium. The attachment of planktonic cells from the surrounding environment plays a prominent role in the initial phase of biofilm lifecycle as it initiates the colony formation. During the maturation stage, biofilms harbor numerous microenvironments which lead to metabolic heterogeneity. Such microniches provide conditions suitable for the growth of new species, which are present in the bulk liquid as planktonic cells and can penetrate the porous biofilm matrix. We present a 1D continuum model on the interaction of sessile and planktonic phenotypes in biofilm lifestyle. Such a model is able to reproduce the key role of planktonic cells in the formation and development of biofilms by considering the initial attachment and colonization phenomena. The model is formulated as a hyperbolic-elliptic free boundary value problem with vanishing initial value which considers the concentrations of planktonic and sessile cells as state variables. Hyperbolic equations reproduce the transport and growth of sessile species, while elliptic equations model the diffusion and conversion of planktonic cells and dissolved substrates. The attachment is modelled as a continuous, deterministic process which depends on the concentrations of the attaching species. The growth of new species is modelled through a reaction term in the hyperbolic equations which depends on the concentration of planktonic species within the biofilm. Existence and uniqueness of solutions are discussed and proved for the attachment regime. Finally, some numerical examples show that the proposed model correctly reproduces the growth of new species within the biofilm and overcomes the ecological restrictions characterizing the Wanner-Gujer type models.
	
\end{abstract}

\maketitle

\section{Introduction} \label{n1}

In recent years, the study of how the sessile and planktonic phenotypes interact in biofilm lifestyle has become a theme of intense interest and scrutiny \cite{biofilm}. Biofilms are microbial assemblies which commonly develop attached to abiotic or biotic surfaces. They are characterized by a solid matrix of extracellular polymeric substance (EPS) in which microorganisms are embedded \cite{flemming2010}. The biofilm dynamics are deeply influenced by microbial mass exchanges between biofilm and the surrounding environment, which involve both the sessile and planktonic biomasses. The biofilm formation is initiated by pioneer microbial cells in planktonic form, which attach to a solid support through an initial attachment process. Such cells switch their mode of growth from planktonic to sessile and constitute the first sessile microbial colony \cite{palmer2007bacterial}, which develops and expands over time as a result of the microbial metabolic growth. Meanwhile, large EPS production by sessile cells confers high density and compactness to the aggregate and protects it from external agents. During the maturation stage, the high density induces large spatial gradients in biofilm properties, leading to numerous microenvironments and extremely heterogeneous microbial distributions. Specifically, new biological conditions arising within the biofilm can promote the phenomenon of microbial invasion: motile planktonic cells colonize the aggregate by penetrating the biofilm matrix, and proliferate as new sessile biomass where ideal conditions for their metabolic activity occur \cite{sutherland2001biofilm}. This means that the number of microbial species constituting the biofilm can increase over time, since microbial species initially not present can join the biofilm when new metabolic microniches arise. Furthermore, external shear forces, nutrients depletion and biomass decay lead to the detachment of cells from the biofilm colony into the surrounding medium \cite{trulear1982dynamics}. Lastly, in the final stage of the biofilm lifecycle, microbial dispersal phenomena can occur: as a result of habitat decay (resource depletion and cell competition for space), planktonic cells, known as dispersed cells, are released in the surrounding environment, migrate to new surfaces and subsequently constitute new biofilm aggregates \cite{mcdougald2012should}. 

Despite the high amount of mathematical works on multispecies biofilms growth developed in the framework of the Wanner and Gujer model \cite{wanner1986multispecies} or as multidimensional partial differential equation models \cite{kla,cogan,ward,eberl,clarelli}, most of them completely neglect the attachment process in the initial phase of biofilm formation, since the initial data that prescribe location, size, and composition of colonies at the onset of the simulations are arbitrarily assigned. This strongly affects the biofilm development and maturation as highlighted by a recent work \cite{eberl2} where the attachment has been incorporated as a discrete stochastic process in a density-dependent diffusion-reaction model for cellulolytic biofilms. Furthermore, the Wanner-Gujer type models \cite{wanner1986multispecies,mavsic2014modeling,eberl3,d2019free} can lead, in some cases, to ecological restrictions on the number of species constituting the biofilm \cite{klapper2011exclusion}. Indeed, they are characterized by a restriction on the number of species that can inhabit the biofilm under the detachment regime: that is if a species is not initially present within the biofilm on the support, it will be washed out from the system. The free boundary problem introduced in this work is intended to overcome these limitations by considering the initial biofilm formation mediated by planktonic cells as well as the colonization process. In particular, we present a one-dimensional continuous model considering two state variables representing the planktonic and sessile phenotypes and reproducing the transition from the former to the other in the biofilm lifecycle. The underlying model is a coupled hyperbolic-elliptic free boundary value problem with nonlocal effects. The attachment is modelled as a continuous, deterministic process which depends on the concentrations of the attaching species in the bulk liquid \cite{d2019free}. The colonization process which results in the establishment of new species in sessile form is modelled by considering an additional reaction term in the hyperbolic equations, which depends on the concentration of planktonic species within the biofilm \cite{d2015modeling}. The concentration of the planktonic species within the biofilm is governed by elliptic partial differential equations which describe their diffusion from the bulk liquid within the biofilm. A reaction term is considered to account for the conversion of the planktonic phenotype into the sessile mode of growth. 

The work is organized as follows. Section \ref{n2} introduces the mathematical background for the attachment process in the initial phase of multispecies biofilm formation, in the framework of the Wanner-Gujer approach to biofilm modelling \cite{d2019free}. The free boundary is constituted by the biofilm thickness and it is assumed to be initially zero. The growth of the attaching species is governed by nonlinear hyperbolic partial differential equations. The free boundary is governed by a first order differential equation that depends on attachment, detachment and biomass growth velocity. It is recalled that the free boundary velocity is greater than the characteristic velocity of the mentioned hyperbolic system during the first instants of biofilm formation. As a consequence, the free boundary is a space-like line. The initial-boundary conditions for the microbial concentrations are assigned on this line and they are equal to the relative abundance of the species in the biomass attached to the biofilm-bulk liquid interface. The free boundary value problem is completed by a system of semi-linear elliptic partial differential equations that governs the quasi-static diffusion of substrates. 
In Section \ref{n3}, a numerical experiment shows that the free boundary problem introduced in \cite{d2019free} needs to be generalized to eliminate any restriction on the number of species inhabiting the mature biofilm as described in \cite{klapper2011exclusion}. Section \ref{n4} introduces the new free boundary problem which accounts for both the initial phase of biofilm formation and the diffusion and colonization of planktonic species within the biofilm. Section \ref{n5} introduces the integral version of the differential free boundary problem provided in Section \ref{n4}, which is derived by adopting characteristics coordinates. An existence and uniqueness theorem of solutions is shown in Section \ref{n5} in the class of continuous functions. The proposed model is also solved numerically to simulate the biofilm evolution during biologically relevant conditions and provides interesting insights towards quantitative understanding of biofilm dynamics and ecology. Numerical results are reported in Section \ref{n6}. Finally, the conclusions of the work are outlined in Section \ref{n7}.

%%%%%%%%%%%%%%%%%%%%%%%%%%%%%%%%%%%%%%%%%%%%%%%%%%%%%%%%%%%%%%%%%%%%%%%%%%
%%%%%%%%%%%%%%%%%%--------NEW SECTION-----------%%%%%%%%%%%%%%%%%%%%%%%%%%
%%%%%%%%%%%%%%%%%%%%%%%%%%%%%%%%%%%%%%%%%%%%%%%%%%%%%%%%%%%%%%%%%%%%%%%%%%

 \section{Background} \label{n2}

 A free boundary approach was introduced in \cite{d2019free} for modelling
 the initial phase of the multispecies biofilm formation and
 growth in the framework of Wanner and Guyer model \cite{wanner1986multispecies}.
 In this context, denoting by $X_i(z,t)$ the concentration of the
 generic bacterial species $i$,
 the one-dimensional multispecies biofilm growth is governed by
 the following system of nonlinear hyperbolic partial differential
 equations
 \begin{equation}                                        
 \frac{\partial}{\partial t}X_i(z,t)+
 \frac{\partial}{\partial z}(u(z,t)X_i(z,t))
 =\rho_ir_{M,i},\ i=1,...,n,       \label{2.1}
 \end{equation}
 where $u(z,t)$ denotes the velocity of the microbial mass, $r_{M,i}$ the specific
 growth rate, and $\rho_i$ the constant density. In addition, the
 substratum is assumed to be placed at $z=0$.

 \noindent
 The function $r_{M,i}$ depends on
 ${\bf X}=(X_1,...,X_n)$, and substrates $S_{j}$, $j=1,...,m$, as well
 \begin{equation}                                        
 r_{M,i}=r_{M,i}({\bf X}(z,t),{\bf S}(z,t)),\
 {\bf S}=(S_1,...,S_m).       \label{2.2}
 \end{equation}

\noindent
  $u(r,t)$ is governed by the following equation:
		\begin{equation}                                       	
	     \frac{\partial u}{\partial z} = \sum_{i=1}^{n}r_{M,i},\ 0< z\leq L(t),\ t>0,\ u(0,t)=0,                            \label{2.3}	
	\end{equation}
	\noindent
	where $L(t)$ represents the biofilm thickness.

 The substrate diffusion is governed by semi-linear parabolic
 partial differential equations that are usually considered in
 quasi-static conditions \cite{mavsic2014modeling}
 \begin{equation}                                        
   -D_j\frac{\partial^2 S_j}{\partial z^2}
   = r_{S,j}({\bf X}(z,t),{\bf S}(z,t)),\ j=1,...,m,            \label{2.4}
 \end{equation}
 where the functions $r_{S,j}$ denote the conversion rate of substrate
 $j$ and $D_j$ the diffusion coefficients assumed constant.

 The biofilm thickness $L(t)$ represents the free boundary of the mathematical
 problem. Its evolution is governed by the following ordinary differential
 equation \cite{wanner1986multispecies,d2019free,coclite1,coclite2},
 \begin{equation}                                        
 \dot L(t)=u(L,t)+\sigma_{a}(\mbox{\boldmath $\psi^*$})-\sigma_d(L),           \label{2.5}
 \end{equation}
 \noindent
 where $\sigma_a$ denotes the attachment velocity of biomass
 from bulk liquid to biofilm
 and $\sigma_d$ the detachment velocity of biomass from biofilm to bulk liquid.
 The function $\sigma_{a}$ depends linearly on the concentrations
 $\psi_i^*$, $i=1,...,n$, $\mbox{\boldmath $\psi$}^*=(\psi_1^*,...,\psi_n^*)$,
 of the microbial species in planktonic form present in the bulk liquid \cite{wanner1986multispecies,mavsic2014modeling,d2019free}. 
 According to the experimental evidence, the ability of colonizing a clean surface is a feature of few 
 microbial species, which are able to switch from their planktonic state, attach to the surface and start 
 to secrete a polymeric matrix anchoring the cells to each other and to the surface. 
 Even the formation of a single layer of cells can lead to a change on the electrostatic 
 nature and mechanical properties of the surface, that can facilitate the attachment 
 of new species that were initially unable to colonize the clean surface.
 According to \cite{d2019free}, this is taken into account by considering in 
 the formulation of the attachment flux $\sigma_{a}=\sum_{i=1}^n v_{a,i}\psi_i^*/\rho_i$
 different attachment velocities $v_{a,i}$ for the single microbial species living in the liquid environment.
 Such velocities can be assigned constant or can be considered as functions of the environmental 
 conditions affecting biofilm growth, that is substrate concentrations, 
 biofilm composition itself, electrostatic and mechanical properties of the surface.

 The function
 $\sigma_{d}$ is usually assumed to be proportional to $L^2$:
 $\sigma_{d}=\delta L^2$, \cite{abbas2012longtime}, where $\delta$ depends on the
 mechanical properties of the biofilm.
 In the initial phase of biofilm formation, where $L(0)=0$, the attachment is the
 prevailing process and $\sigma_{d}$ is very small, since so
 is $L^2$. Therefore, it
 is $\sigma_{a}-\sigma_{d}>0$ and the free boundary velocity is
 greater than the characteristic velocity, $\dot L(t)>u(L,t)$.
 The free boundary is a space-like line, as illustrated in
 Fig. \ref{f2.1}.
 \begin{figure}
 \fbox{\includegraphics[width=1\textwidth, keepaspectratio]{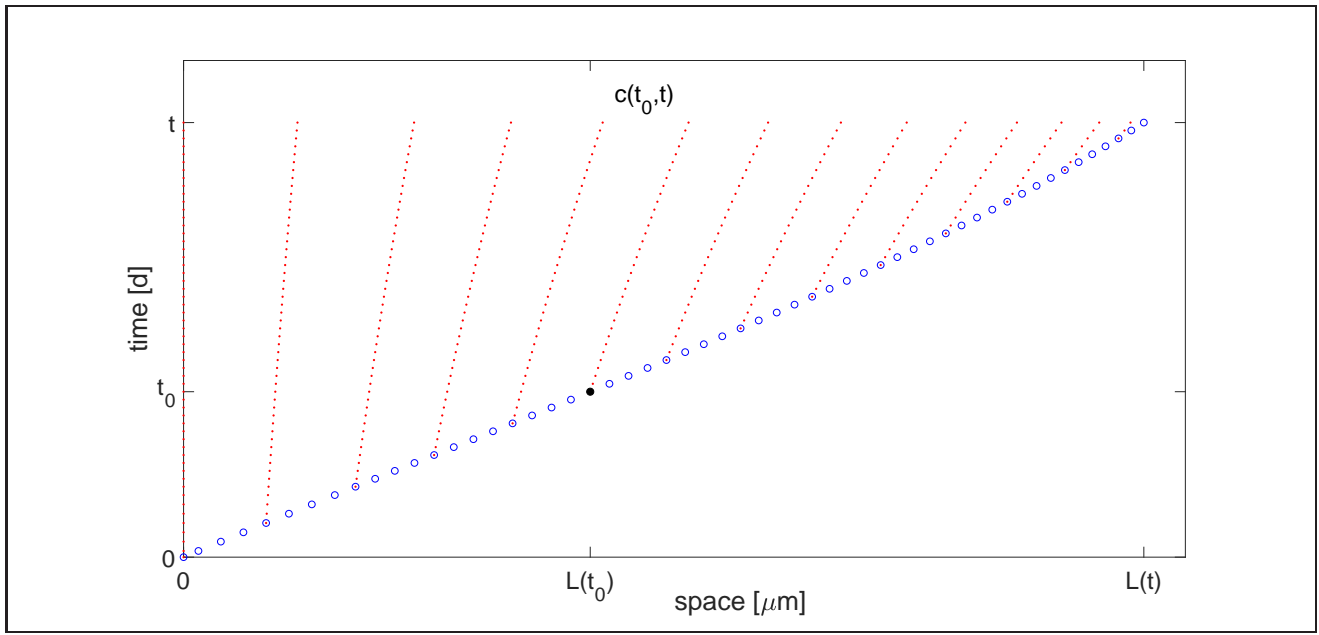}}   
 \caption{Time evolution of the free boundary with vanishing initial value and characteristic 
 lines when $\sigma_{a}-\sigma_{d}>0$. The free boundary is a space-like line. The blue dotted line denotes the free boundary evolution. 
 Red dotted lines denote the characteristic-like lines.} \label{f2.1}
 \end{figure}

 \begin{figure}
 \fbox{\includegraphics[width=1\textwidth, keepaspectratio]{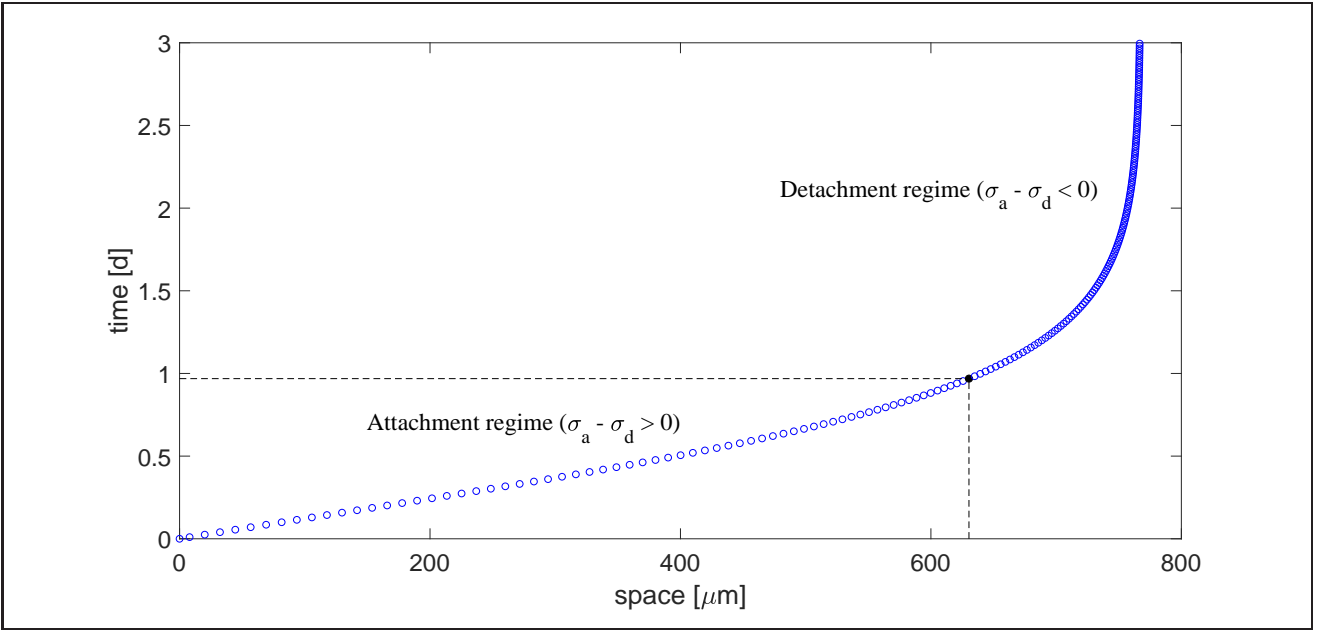}}   
 \caption{Time evolution of the free boundary with vanishing initial value. Note that the biofilm thickness undergoes both 
 attachment ($\sigma_{a}-\sigma_{d}>0$) and detachment regimes ($\sigma_{a}-\sigma_{d}<0$), 
 the latter prevailing for large $L$. The blue dotted line denotes the free boundary evolution.} \label{f2.2}
 \end{figure}
 \noindent
 In the same figure, the characteristic-like lines of system
 \eqref{2.1} are also depicted. These lines, $z=c(t_0,t)$, are defined by the
 differential initial value problem
 \begin{equation}                                        
 \frac{\partial c}{\partial t}(t_0,t)
    =u(c(t_0,t),t),\ \ c(t_0,t_0)=L(t_0).                    \label{2.6}
 \end{equation}
 For mature biofilms the free boundary $L$ becomes large, the
 detachment is the prevailing process, it is $\sigma_{a}-\sigma_{d}<0$
 and the free boundary is a time-like line, Fig. \ref{f2.2}.
 
 The free boundary value problem \eqref{2.1}-\eqref{2.4} was
 discussed in \cite{d2019free} under the following initial-boundary
 conditions:

 \begin{equation}                                        
 X_{i}(L(t),t)=X_{i,0}(t),\ i=1,...,n,                   \label{2.7}
 \end{equation}
 \begin{equation}                                       
   \frac{\partial S_j}{\partial z}(0,t)=0,\
   S_j(L,t)=S_j^*(t),\ j=1,...,m,                           \label{2.8}
 \end{equation}
 \begin{equation}                                        
  L(0)=0.                                                \label{2.9}
 \end{equation}
\noindent
 In equations \eqref{2.7}, $X_{i,0}(t)$ is the relative abundance 
 of the species $i$ in the biomass attached to the biofilm-bulk liquid interface \cite{wanner1996mathematical}.
 More precisely, $X_{i,0}(t)$ can be evaluated as 
 \begin{equation}                                      
 X_{i,0}(t)=\frac{v_{a,i}\psi_i^*(t)}{\sum_{i=1}^n v_{a,i}\psi_i^*(t)}\rho_i,\ i=1,...,n,          \label{2.10}
 \end{equation}
 \noindent
 where $\sigma_{a,i}=v_{a,i}\psi_i^*(t)$ denotes the attachment flux of the single species $i$ 
 and $\sigma_a=\sum_{i=1}^n v_{a,i}\psi_i^*(t)$ the total attachment flux.
 According to \eqref{2.10}, the concentration of the microbial species at the biofilm-liquid interface 
 $X_i(L(t),t)$ for a multispecies biofilm growing under attachment regime, 
 depends on both the concentrations of the same species in planktonic form
 in the bulk liquid and their attachment propensity. Note that, when all the microbial 
 species in the bulk liquid are characterized by the same attachment velocity, equations \eqref{2.10} reduces to
 \[
 \frac{X_{i,0}(t)}{\rho_i}=\frac{\psi_i^*(t)}{\sum_{i=1}^n\psi_i^*(t)},
 \]
 that is the volume 
 fraction of the microbial species $i$ at the biofilm-bulk liquid interface
 assumes the same value of the volume fraction within the bulk liquid. This reproduces 
 the case of a biofilm that will be initially constituted by all microbial 
 species inhabiting the surrounding liquid environment. However, 
 going on with time the biofilm composition is affected by other factors 
 such as substrate availability, specific microbial growth rate, detachment flux.

 For what concerns substrate diffusion, the first boundary condition \eqref{2.8} is the no flux condition at
 substratum. The functions $S_j^*(t)$ in the second boundary
 condition \eqref{2.8} are prescribed functions in general.

%%%%%%%%%%%%%%%%%%%%%%%%%%%%%%%%%%%%%%%%%%%%%%%%%%%%%%%%%%%%%%%%%%%%%%%%%%%
%%%%%%%%%%%%%%%%%%%--------NEW SECTION-----------%%%%%%%%%%%%%%%%%%%%%%%%%%
%%%%%%%%%%%%%%%%%%%%%%%%%%%%%%%%%%%%%%%%%%%%%%%%%%%%%%%%%%%%%%%%%%%%%%%%%%%

 \section{Criticism} \label{n3}

As outlined in \cite{d2019free}, the model for
 the initial biofilm formation, summarized in the previous section,
 should be generalized to include the possibility that new attaching
 bacterial species can move downward within the biofilm matrix and colonize the
 regions where the conditions for their growth are optimal. An example, referred to as \textit{Case 1}, could help to better understand the question. 
 To discuss this special problem, an
 equivalent expression will be used for equations \eqref{2.1},
 where $X_{i}$ is replaced by the volume fraction $f_i$ defined by
 \begin{equation}                                      
 f_i(z,t)=X_{i}(z,t)/\rho_i,\ i=1,...,n,               \label{3.1}    
 \end{equation}              
 subjected to the constraint
 \begin{equation}                                         
  \sum_{i=1}^{n}f_i(z,t)=1.                    \label{3.2}
 \end{equation}
 Considering \eqref{3.1} in \eqref{2.1} yields
 \begin{equation}                                        
 \frac{\partial}{\partial t}f_i(z,t)+
 \frac{\partial}{\partial z}(u(z,t)f_i(z,t))
 =r_{M,i},\ i=1,...,n.                          \label{3.3}
 \end{equation}
 For equations above, conditions \eqref{2.7} are replaced by
 \begin{equation}                                        
 f_{i}(L(t),t)=f_{i,0}(t),\ i=1,...,n,                       \label{3.4}
 \end{equation}
 where
 \begin{equation}                                        
 f_{i,0}(t)=X_{i,0}(t)/\rho_i.                             \label{3.5}
 \end{equation}

 Let us consider a three species and substrate biofilm $n=3, \ m=3$ growing under time-dependent
 conditions. In particular, the model simulates the case of a biofilm growing in 
 a liquid environment initially inhabited by species $\psi_1^*$ and $\psi_2^*$ and 
 continuously fed with substrates $S_1$ and $S_2$. At time $t=t_1>0$, a third species $\psi_3^*$ is supposed to
 be fed into the system

 \begin{equation}                                        
 \psi_i^*(t)=\psi_{i,0}^*> 0,\ 0\leq t\leq T,\ i=1,2,                  \label{3.6}
 \end{equation}
 \begin{equation}                                        
 \psi_3^*(t)=\left\{
   \begin{array}{ll}
    0, & 0\leq t\leq t_1, \\
    \psi_{3,0}^* \frac{(t-t_1)^{10}}{t_1^{10/t_1}+(t-t_1)^{10}}>0 & t_1< t\leq T. \\
  \end{array}
 \right.                                          \label{3.7}
 \end{equation}
\noindent
 Species $\psi_1^*$ and $\psi_2^*$ start to attach
 at $t=0$ while the third at $t=t_1>0$

 \begin{equation}                                        
 f_{i,0}(t)>0,\  i=1,2,\ f_{3,0}(t)=0,\ 0\leq t< t_1,                     \label{3.8}
 \end{equation}
 \begin{equation}                                        
 f_{i,0}(t)>0,\  i=1,2,3,\ t \geq t_1.                                    \label{3.9}
 \end{equation}
 \noindent
 Functions $f_{i,0}(t)$ can be derived from equations \eqref{2.10}, \eqref{3.6} and \eqref{3.7}.
 Species $f_1$ and $f_2$ grow on substrate $S_1$ and $S_2$, respectively. Species $f_1$ by consuming substrate $S_1$ produces $S_3$, 
 which is uptaken by $f_3$. All species are supposed to grow only in sessile form, and the reactor is considered as an 
 infinite reserve of
 substrates and planktonic species ($S_j^*(t)=S_{j,0}^*$). 
 The reaction terms $r_{M,i}$ and $r_{S,j}$ in equations \eqref{3.3} and \eqref{2.4} are modelled by using
 Monod type kinetics and are expressed as
  \begin{equation}                                        
  r_{M,1} = \mu_{\max,1} \frac{S_1}{K_1+S_1}  f_1,\
  r_{M,2} = \mu_{\max,2} \frac{S_2}{K_2+S_2}  f_2,\
  r_{M,3} = \mu_{\max,3} \frac{S_3}{K_3+S_3}  f_3,                \label{3.10}
 \end{equation}

 \begin{equation}                                       
 r_{S,1} = -\frac{r_{M,1}}{Y_1} \rho_1,\
 r_{S,2} = -\frac{r_{M,2}}{Y_2} \rho_2,\
 r_{S,3} = \frac{r_{M,1}}{Y_1} \rho_1-\frac{r_{M,3}}{Y_3} \rho_3.             \label{3.11}
 \end{equation}
 \noindent
 The values of the kinetic parameters and boundary conditions used in the numerical simulations are reported in Table \ref{t3.1}.
 All the sessile species are supposed to have the same density $\rho_i=\rho, i=1,...,n$. The simulation time adopted for the numerical 
 experiment is $t=10 \ d$. We are aware that such simulation time will cover both the initial biofilm formation and the maturation phase where 
 the detachment will be prevalent on the attachment flux. This choice is justified by the fact that we were interested in showing 
 also the mature biofilm configuration, which is achieved under detachment regime.

 \begin{table}[ht]
 \begin{footnotesize}
 \begin{center}
 \begin{tabular}{llccc}
 \hline
 {\textbf{Parameter}} & {\textbf{Definition}} & {\textbf{Unit}} & {\textbf{Value}}
 \\
 \hline
 $\mu_{max,1}$   & Maximum specific growth rate for $f_1$           &  $d^{-1}$                              & $0.4$           \\
 $\mu_{max,2}$   & Maximum specific growth rate for $f_2$           &  $d^{-1}$                              & $1.5$           \\
 $\mu_{max,3}$   & Maximum specific growth rate for $f_3$           &  $d^{-1}$                              & $0.5$           \\
 $K_{1}$         & Half saturation constant for $f_1$ on $S_1$      &  $g \ m^{-3}$                          & $1$             \\
 $K_{2}$         & Half saturation constant for $f_2$ on $S_2$      &  $g \ m^{-3}$                          & $20$            \\
 $K_{3}$         & Half saturation constant for $f_3$ on $S_3$      &  $g \ m^{-3}$                          & $1$             \\
 $Y_{1}$         & Yield of $f_1$ on $S_1$                          &  $--$                                  & $0.4$           \\
 $Y_{2}$         & Yield of $f_2$ on $S_2$                          &  $--$                                  & $0.9$           \\
 $Y_{3}$         & Yield of $f_1$ on $S_1$                          &  $--$                                  & $0.9$           \\
 $D_{1}$         & Diffusion coefficient of $S_1$ in biofilm        &  $m^2 \ d^{-1}$                        & $10^{-5}$       \\
 $D_{2}$         & Diffusion coefficient of $S_2$ in biofilm        &  $m^2 \ d^{-1}$                        & $10^{-5}$       \\
 $D_{3}$         & Diffusion coefficient of $S_3$ in biofilm        &  $m^2 \ d^{-1}$                        & $10^{-5}$       \\
 $\rho$          & Biofilm density                                  &  $g \ m^{-3}$                          & $5000$            \\
 $\delta$        & Biomass shear constant                           &  $m^{-1} \ d^{-1}$                     & $2000$              \\
 $S_{1,0}^*$       & $S_1$ concentration in the bulk liquid         &  $g \ m^{-3}$                          & $100$           \\
 $S_{2,0}^*$       & $S_2$ concentration in the bulk liquid         &  $g \ m^{-3}$                          & $100$           \\
 $S_{3,0}^*$       & $S_3$ concentration in the bulk liquid         &  $g \ m^{-3}$                          & $0$             \\
 $\psi_{1,0}^*$    & $\psi_1^*$ concentration in the bulk liquid    &  $g \ m^{-3}$                          & $100$           \\
 $\psi_{2,0}^*$    & $\psi_2^*$ concentration in the bulk liquid    &  $g \ m^{-3}$                          & $100$           \\
 $\psi_{3,0}^*$    & $\psi_3^*$ concentration in the bulk liquid    &  $g \ m^{-3}$                          & $100$             \\
 $v_{a,1}$         & $\psi_1^*$ attachment velocity                 &  $m \ d^{-1}$                          & $2.5\cdot10^{-2}$             \\
 $v_{a,2}$         & $\psi_2^*$ attachment velocity                 &  $m \ d^{-1}$                          & $2.5\cdot10^{-2}$             \\
 $v_{a,3}$         & $\psi_3^*$ attachment velocity                 &  $m \ d^{-1}$                          & $2.5\cdot10^{-2}$             \\ 
 \hline
 \end{tabular}\\
 \caption{Kinetic parameters used for model simulations} \label{t3.1}
 \end{center}
 \end{footnotesize}
 \end{table}

 Fig. \ref{f3.1} shows the free boundary evolution and the characteristic line $c(t_1,t)$ starting from $(L(t_1),t_1)$ 
 up to $0.8d$ simulation time. Figs. \ref{f3.2} and \ref{f3.3} show the biofilm composition and substrate trends within the biofilm
 over time, under attachment and detachment regimes respectively.

 \begin{figure}
 \fbox{\includegraphics[width=0.60\textwidth, keepaspectratio]{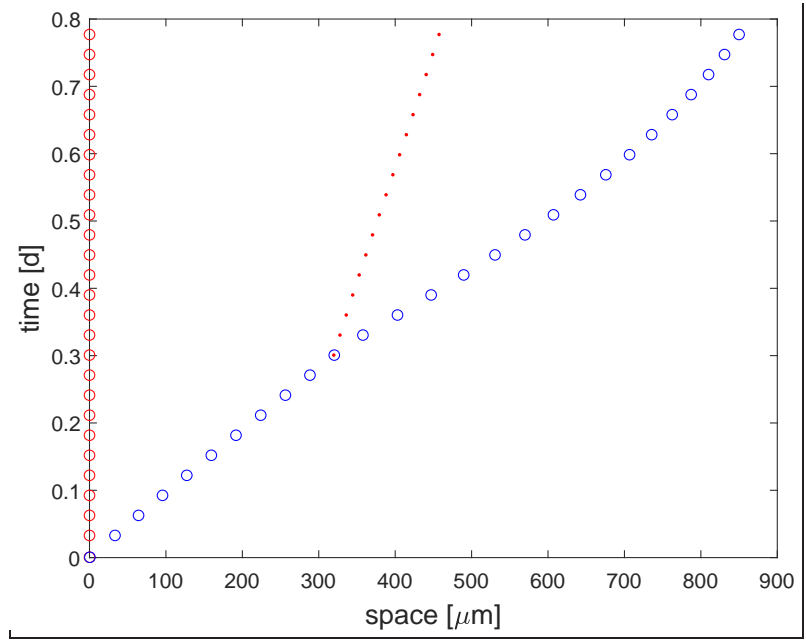}}   
 \caption{Time evolution of the free boundary (blue open dots) and the characteristic line $c(t_1,t)$ (red solid dots) under attachment regime ($\sigma_{a}-\sigma_{d}>0$). 
 Red open dots denote the characteristic line $c(0,t)$.}\label{f3.1}
 \end{figure}  

 \begin{figure} 
 \fbox{\includegraphics[width=0.7\textwidth, keepaspectratio]{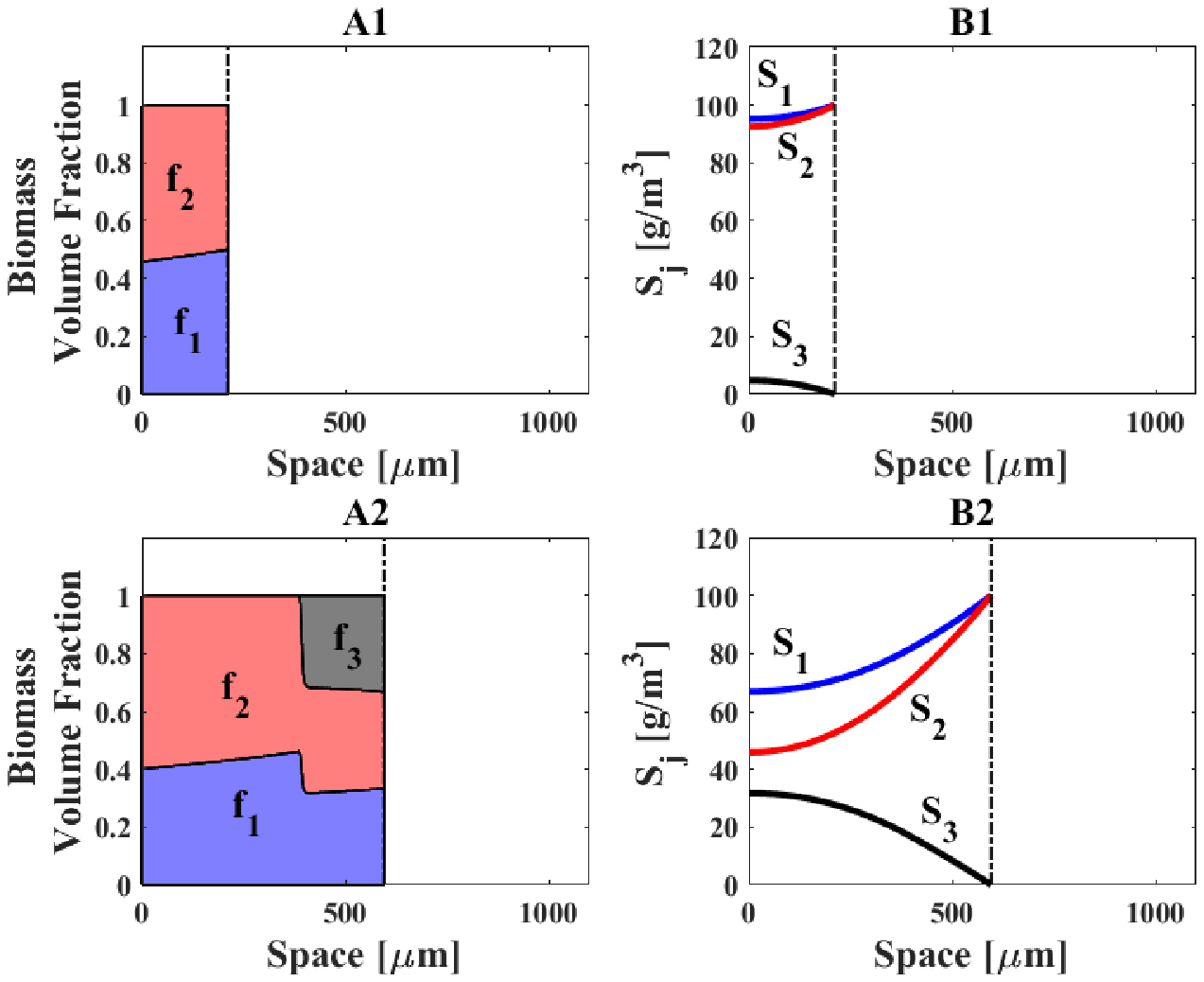}}   
 \caption{Biofilm composition (A1-A2) and substrate distribution (B1-B2) for \textit{Case 1}, under attachment regime, at time
 $t=0.25 \ d$ (top) and $t=0.50 \ d$ (bottom).} \label{f3.2}
 \end{figure}   

 \begin{figure}
 \fbox{\includegraphics[width=0.7\textwidth, keepaspectratio]{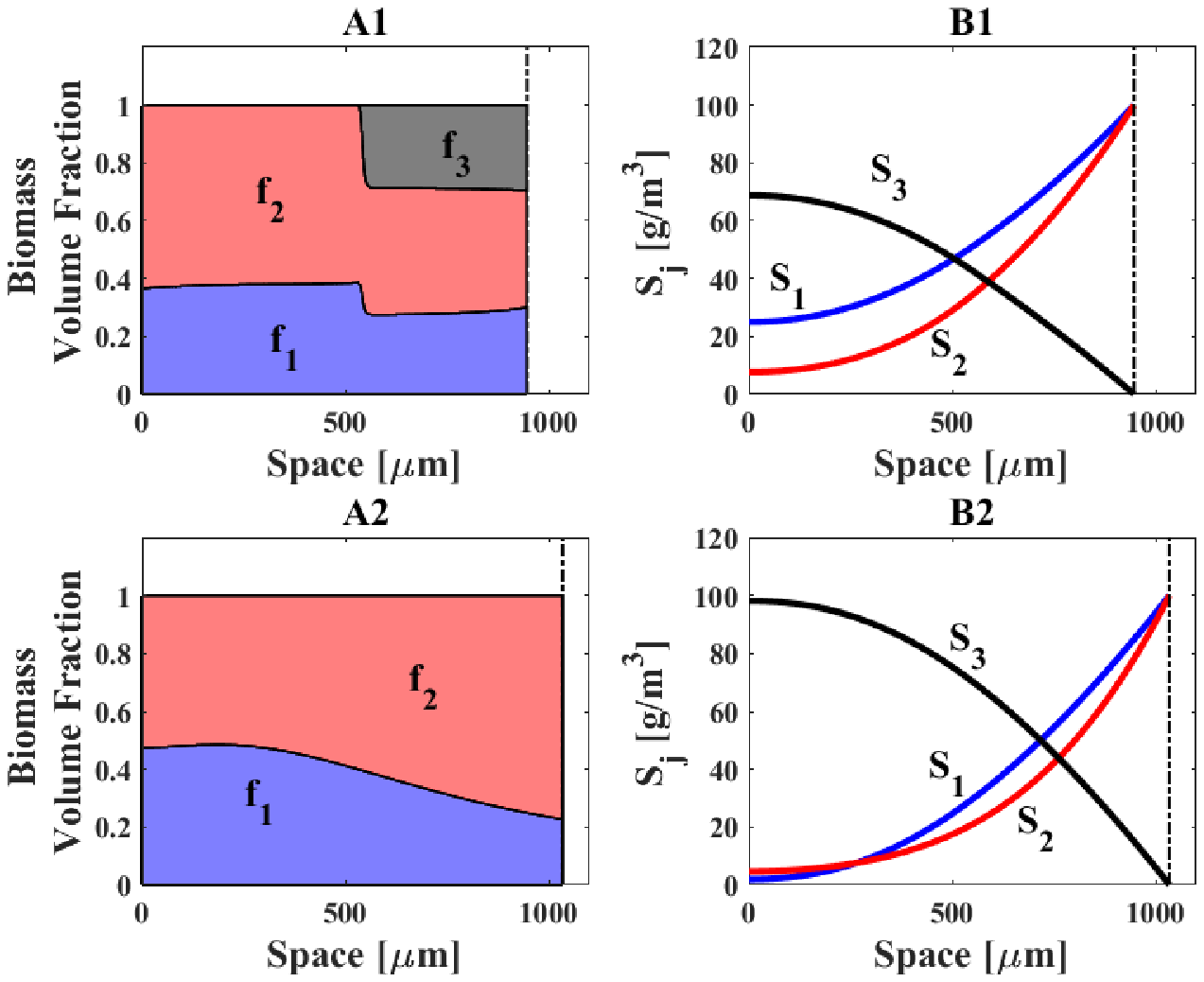}}   
 \caption{Biofilm composition (A1-A2) and substrate distribution (B1-B2) for \textit{Case 1}, under detachment regime, at time
 $t=1 \ d$ (top) and $t=10 \ d$ (bottom).}     \label{f3.3}
 \end{figure}

 \noindent
 The third species begins to adhere to the biofilm-bulk liquid interface at $t=t_1$. 
 Substrates $S_1$ and $S_2$ are consumed within biofilm by species $f_1$ and $f_2$. As a consequence, favorable conditions for $f_3$ growth occurs 
 within the inner biofilm region due to $S_3$ production. According to the uniqueness
 and existence theorem provided in \cite{d2019free}, the third species is confined
 within the region $z>c(t_1,t)$ and does not colonize the region $z< c(t_1,t)$
 where there are favorable conditions for its growth
  \begin{equation}                                        
  f_3(z,t)\left\{
   \begin{array}{lll}
    =0, & 0\leq z< c(t_1,t), & 0\leq t< t_1, \\
    >0, & z\geq c(t_1,t), & t\geq t_1. \\
  \end{array}
 \right.                                             \label{3.12}
 \end{equation}
 \noindent
 As shown in Figs. \ref{f3.2} and \ref{f3.3}, the third species is unable to
 penetrate the inner biofilm region. 
 Under detachment regime, its concentration tends to zero and it is completely 
 washed out from the biofilm system. This is a well-known
 behavior of the model \cite{wanner1986multispecies}, as stated in \cite{klapper2011exclusion}.
 \noindent
 The model introduced in this work is intended to eliminate this great limitation.

%%%%%%%%%%%%%%%%%%%%%%%%%%%%%%%%%%%%%%%%%%%%%%%%%%%%%%%%%%%%%%%%%%%%%%%%%%%
%%%%%%%%%%%%%%%%%%%--------NEW SECTION-----------%%%%%%%%%%%%%%%%%%%%%%%%%%
%%%%%%%%%%%%%%%%%%%%%%%%%%%%%%%%%%%%%%%%%%%%%%%%%%%%%%%%%%%%%%%%%%%%%%%%%%%

 \section{Statement of the free boundary problem} \label{n4}

This section presents the free boundary value problem for
 biofilm growth which considers its initial formation ($\sigma_a-\sigma_d>0, L(0)=0$) 
 and the diffusion and colonization of the planktonic species
 within the biofilm. It is a generalization of the problem discussed
 in \cite{mavsic2014modeling,d2019free} and it is obtained by developing some ideas introduced in \cite{d2015modeling}. 
 Specifically, an additional state variable is considered, $\Psi_i$, ${\bf \Psi}=(\Psi_1,..., \Psi_n)$, 
 which represents the concentration of the planktonic species within the biofilm. The differential mass balance equations \eqref{2.1} are modified
 by adding a growth rate term $r_i$ that takes into account for
 the colonizing bacterial species, and further equations are introduced
 for the diffusion of the planktonic species. The resulting model is able
 to overcome the criticism outlined in Sec. \ref{n3}, as shown through the simple examples reported in Sec. \ref{n6}.

 The biofilm growth is governed by the following equations 
  \begin{equation}                                        \label{4.1}
 \frac{\partial X_i}{\partial t}+
 \frac{\partial}{\partial z}(uX_i)=\rho_ir_{M,i}({\bf X},{\bf S})
 +\rho_ir_{i}({\bf \Psi},{\bf S}),\ 0\leq z\leq L(t),
 \ t>0,i=1,...,n,
 \end{equation}
 \begin{equation}                                        \label{4.2}
 X_{i}(L(t),t)=X_{i,0}(t),\ t>0,\ i=1,...,n,
 \end{equation}
 \begin{equation}                                        \label{4.3}
 \dot L(t)=u(L(t),t)+\sigma_{a}(\mbox{\boldmath $\psi^*$}),
 \ t>0,\ L(0)=0,
 \end{equation}
 \begin{equation}                                        \label{4.4}
   \frac{\partial u}{\partial z}(z,t)=G({\bf X}(z,t),{\bf S}(z,t),{\bf \Psi}(z,t)),
   \ 0< z\leq L(t),\ u(0,t)=0,
 \end{equation}
 where
 \begin{equation}                                        \label{4.5}
   G({\bf X}(z,t),{\bf S}(z,t),{\bf \Psi}(z,t))=
   \sum_{i=1}^{n}(r_{M,i}+r_{i}),
 \end{equation}
\begin{equation}                                        \label{4.6}
   -D_j\frac{\partial^2 S_j}{\partial z^2}
   = r_{S,j}({\bf X}(z,t),{\bf S}(z,t)),\ 0< z< L(t),
 \ t>0,\ j=1,...,m,
 \end{equation}
 \begin{equation}                                        \label{4.7}
   \frac{\partial S_j}{\partial z}(0,t)=0,\
   S_j(L,t))=S_j^*(t),\ t>0,\ j=1,...,m.
 \end{equation}
\noindent
 The diffusion of the colonizing species within the biofilm is
 governed by semi-linear parabolic
 partial differential equations that are considered in
 quasi-static conditions
 \begin{equation}                                        \label{4.8}
   -D_{\Psi,i}\frac{\partial^2 \Psi_i}{\partial z^2}
   = r_{\Psi,i}({\bf \Psi}(z,t),{\bf S}(z,t)),\ 0< z< L(t),
 \ t>0,\ i=1,...,n,
 \end{equation}
 \noindent
 where $r_{\Psi,i}$ indicates the conversion rate due to the switch from planktonic to 
 sessile mode of growth and $D_{\Psi,i}$ is the diffusivity coefficient of the planktonic species within the biofilm. 
 Diffusion equations for ${\bf \Psi}$ are considered in quasi-static conditions for the same reason as ${\bf S}$.
 Equations \eqref{4.8} are integrated with the following Neumann-Dirichlet
 boundary conditions
 \begin{equation}                                        \label{4.9}
   \frac{\partial \Psi_i}{\partial z}(0,t)=0,\
   \Psi_i(L,t)=\psi_i^*(t),\ t>0,\ i=1,...,n,
 \end{equation}
 where the no flux boundary conditions on the support are evident
 and the Dirichlet boundary conditions state that the values of the planktonic species
 on the free boundary are the same as in the bulk liquid.

 Note that equation \eqref{4.3} refers to the initial phase of biofilm formation, 
 when the detachment flux $\sigma_d$ is negligible compared to $\sigma_a$. The free boundary $L(t)$
 is a space-like line and equation \eqref{4.2} provides the initial conditions for the microbial species
 in sessile form on the free boundary. Conversely, during the maturation stage of biofilm growth 
 the detachment flux is predominant and the free boundary is represented by a time-like line as stated in Sec. \ref{n2}.
 The free boundary value problem referring to the mature phase of biofilm growth and considering the 
 interaction between the planktonic and sessile phenotype through the colonization process
 has been investigated both qualitatively and numerically in \cite{d2018moving,d2018invasion}.

%%%%%%%%%%%%%%%%%%%%%%%%%%%%%%%%%%%%%%%%%%%%%%%%%%%%%%%%%%%%%%%%%%%%%%%%%%%
%%%%%%%%%%%%%%%%%%%--------NEW SECTION-----------%%%%%%%%%%%%%%%%%%%%%%%%%%
%%%%%%%%%%%%%%%%%%%%%%%%%%%%%%%%%%%%%%%%%%%%%%%%%%%%%%%%%%%%%%%%%%%%%%%%%%%

 \section{Uniqueness and existence of solutions} \label{n5}

According to \cite{d2019free}, the differential free boundary problem
 \eqref{4.1}-\eqref{4.9} can be converted to an equivalent system of integral equations by using the
 characteristics introduced in \eqref{2.6}. The integral problem is summarized below by
 using the following positions
  \begin{equation}                                        \label{5.1}
     {\bf x}(t_0,t)={\bf X}(c(t_0,t),t),\ \ {\bf x}(x_1,...,x_n),
 \end{equation}
 \begin{equation}                                        \label{5.2}
     {\bf s}(t_0,t)={\bf S}(c(t_0,t),t),\ \ {\bf s}(s_1,...,s_m),
 \end{equation}
 \begin{equation}                                        \label{5.3}
     \mbox{\boldmath $\psi$}(t_0,t)=\mbox{\boldmath $\Psi$}(c(t_0,t),t),\ \ \mbox{\boldmath $\psi$}(\psi_1,...,\psi_n),
 \end{equation}
 The integral equations for $x_i$ follow from \eqref{2.6},\eqref{4.1}-\eqref{4.5}
 \begin{equation}                                        \label{5.4}
  x_i(t_0,t)=X_{i,0}(t_0)+\int_{t_0}^t
  F_i({\bf x}(t_0,\tau),{\bf s}(t_0,\tau),\mbox{\boldmath $\psi$}(t_0,\tau))d\tau,
  \ 0\leq t_0< t\leq T, \ i=1,...,n.
 \end{equation}
 The integral equations for $s_j$ follow from \eqref{4.6}-\eqref{4.7}
 \[
    s_j(t_0,t)=
   \int_{t_0}^{t}   d\theta\int_{0}^{\theta}
   F_{s,j}({\bf x}(\tau,t),{\bf s}(\tau,t),
   \frac{\partial c}{\partial \theta}(\theta,t),
   \frac{\partial c}{\partial \tau}(\tau,t))
   d\tau
 \]
 \begin{equation}                                        \label{5.5}
   +  S_j^*(t),\ \ 0< t_0< t\leq T, \ j=1,...,m,
 \end{equation}
\noindent
 where $F_{s,j}$ is defined in \eqref{5.13} at the end of this
 section.

 \noindent
 Similarly to ${\bf s}$, the integral equations for $\psi_i$ follow from \eqref{4.8}-\eqref{4.9} and write
  \[
    \psi_i(t_0,t)=
   \int_{t_0}^{t}   d\theta\int_{0}^{\theta}
   F_{\psi,i}(\mbox{\boldmath $\psi$}(\tau,t),{\bf s}(\tau,t),
   \frac{\partial c}{\partial \theta}(\theta,t),
   \frac{\partial c}{\partial \tau}(\tau,t))
   d\tau
 \]
 \begin{equation}                                        \label{5.6}
   +  \psi_i^*(t),\ \ 0< t_0< t\leq T, \ i=1,...,n,
 \end{equation}
\noindent
 where $F_{\psi,i}$ is defined in \eqref{5.14}.
\noindent
 The integral equation for $L$ follows from \eqref{2.6},\eqref{4.3},\eqref{4.4}
 \begin{equation}                                        \label{5.7}
   L(t_0)=\Sigma(t_0)+\int_{0}^{t_0}\ d\theta \int_{0}^{\theta}
 F_{L}({\bf x}(\tau,\theta),{\bf s}(\tau,\theta),\mbox{\boldmath $\psi$}(\tau,\theta),
   \frac{\partial c}{\partial \tau}(\tau,\theta))
  d\tau, \ 0<t_0\leq T,
 \end{equation}
 with $\Sigma(t_0)$ and $F_{L}$ defined in \eqref{5.15}-\eqref{5.16}.
\noindent
 The integral equations for $c(t_0,t)$ and $\partial c/\partial t_0$
 can be obtained from \eqref{2.6},\eqref{4.3}-\eqref{4.5} rewritten in terms of characteristic coordinates
 \[
 c(t_0,t)=\Sigma(t_0)
 + \int_{0}^{t_0}d\theta
 \int_{0}^{\theta}
 F_{c,1}({\bf x}(\tau,\theta),{\bf s}(\tau,\theta),\mbox{\boldmath $\psi$}(\tau,\theta),
  \frac{\partial c}{\partial \tau}(\tau,\theta)) d\tau
 \]
 \begin{equation}                                        \label{5.8}
 +\int_{t_0}^{t}d\theta\int_{0}^{t_0}
 F_{c,1}({\bf x}(\tau,\theta),{\bf s}(\tau,\theta),\mbox{\boldmath $\psi$}(\tau,\theta),
   \frac{\partial c}{\partial \tau}(\tau,\theta))
 d\tau,\ \ 0< t_0< t\leq T,
 \end{equation}

 \[
 \frac{\partial c}{\partial t_0}(t_0,t)
   =\int_{t_0}^{t}
   F_{c,2}({\bf x}(t_0,\theta),{\bf s}(t_0,\theta),\mbox{\boldmath $\psi$}(t_0,\theta),
   \frac{\partial c}{\partial t_0}(t_0,\theta))
   d\theta
 \]
 \begin{equation}                                        \label{5.9}
 +\sigma_{a}(\mbox{\boldmath $\psi^*$}(t_0)),
   \ \ 0< t_0< t\leq T,
 \end{equation}
 where

 \begin{equation}                                        \label{5.10}
 F_{c,1}({\bf x}(\tau,\theta),{\bf s}(\tau,\theta),\mbox{\boldmath $\psi$}(\tau,\theta),
   \frac{\partial c}{\partial \tau}(\tau,\theta))
=G({\bf x}(\tau,\theta),{\bf s}(\tau,\theta),\mbox{\boldmath $\psi$}(\tau,\theta))
 \frac{\partial c}{\partial \tau}(\tau,\theta),
 \end{equation}

 \begin{equation}                                        \label{5.11}
 F_{c,2}({\bf x}(t_0,\theta),{\bf s}(t_0,\theta),\mbox{\boldmath $\psi$}(t_0,\theta),
   \frac{\partial c}{\partial t_0}(t_0,\theta))=G({\bf x}(t_0,\theta),{\bf s}(t_0,\theta),\mbox{\boldmath $\psi$}(t_0,\theta))
   \frac{\partial c}{\partial t_0}(t_0,\theta).
 \end{equation}
\noindent
 The functions introduced in equations \eqref{5.4}-\eqref{5.7} are defined below
 \begin{equation}                                        \label{5.12}
  F_i=\rho_i(r_{M,i}+r_i)-X_iG,\ \ i=1,...,n,
 \end{equation}

 \begin{equation}                                       
     F_{s,j}({\bf x}(\tau,t),{\bf s}(\tau,t),
   \frac{\partial c}{\partial \theta}(\theta,t),
   \frac{\partial c}{\partial \tau}(\tau,t)) =D_j^{-1}
   r_{S,j}({\bf x}(\tau,t),{\bf s}(\tau,t))
   \frac{\partial c}{\partial \theta}(\theta,t)
   \frac{\partial c}{\partial \tau}(\tau,t),          \label{5.13}
 \end{equation}

 \begin{equation}                                        
    F_{\psi,i}(\mbox{\boldmath $\psi$}(\tau,t),{\bf s}(\tau,t),
   \frac{\partial c}{\partial \theta}(\theta,t),
   \frac{\partial c}{\partial \tau}(\tau,t)) = D_{\psi,i}^{-1} \
   r_{\psi,i}(\mbox{\boldmath $\psi$}(\tau,t),{\bf s}(\tau,t))
   \frac{\partial c}{\partial \theta}(\theta,t)
   \frac{\partial c}{\partial \tau}(\tau,t),     \label{5.14}
 \end{equation}

 \begin{equation}                                        \label{5.15}
   \Sigma(t_0)=\int_{0}^{t_0}\sigma_{a}(\mbox{\boldmath $\psi^*$}(\theta))
 d\theta,
 \end{equation}

 \begin{equation}                                        \label{5.16}
 F_{L}({\bf x}(\tau,\theta),{\bf s}(\tau,\theta),\mbox{\boldmath $\psi$}(\tau,\theta)
   \frac{\partial c}{\partial \tau}(\tau,\theta))=G({\bf x}(\tau,\theta),{\bf s}(\tau,\theta),\mbox{\boldmath $\psi$}(\tau,\theta))
 \frac{\partial c}{\partial \tau}(\tau,\theta).
 \end{equation}
\noindent
An existence and uniqueness theorem for the integral problem \eqref{5.4}-\eqref{5.9}
can be proved in the space of the continuous functions as generalization of the results in \cite{d2019free}.

 \begin{thm}\label{thm}
  Suppose that:

 (a) $x_{i}^{}(t_0^{},t),s_{j}^{}(t_0^{},t),\psi_{i}^{}(t_0^{},t),c(t_0^{},t),c_{t_0}^{}(t_0^{},t)
 \in C^0([0,\ T_1]\times[0,\  T_1])$,
 \ $T_1>0$, \ $i=1,...,n$, \ $j=1,...,m$,\ and
 $L(t_0^{})\in C^0([0,\ T_1])$;

 (b) $X_{i,0}(t_0^{}), \sigma_{a}(\mbox{\boldmath $\psi^*$}(t_0)), S_j^*(t), \Psi_i^*(t)
 \in C^0([0,\ T_1]),
 \ i=1,...,n, \ j=1,...,m$;

 (c) $|x_{i}^{}-X_{i,0}|\leq h_{x,i}, \ i=1,...,n$;
 $|s_{j}^{}-S_j^*|\leq h_{s,j}, \ j=1,...,m$;
 $|\psi_{i}^{}-\psi_i^*|\leq h_{\psi,i}, \ i=1,...,n$;
 $|L-\Sigma|\leq h_{L}$; $|c-\Sigma|\leq h_{c,1}$;
 $|c_{t_0}^{}-\sigma_a|\leq h_{c,2}$,
 where $h_{x,i},h_{s,j},h_{\psi,i},h_{L},h_{c,1},h_{c,2}$ are positive constants;

 (d) $F_i, i=1,...,n$, $F_{s,j}, j=1,...,m$, $F_{\psi,i}, i=1,...,n$, 
 $F_{L},F_{c,1},F_{c,2}$ are bounded and Lipschitz continuous with respect to their arguments
 \[
 M_i=\max |F_i|,\ i=1,...,n,\ M_{s,j}=\max |F_{s,j}|,\ j=1,...,m,
 \]
 \[
  M_{\psi,i}=\max |F_{\psi,i}|,\ i=1,...,n,\ M_{L}=\max |F_{L}|,\ M_{c,1}=\max |F_{c,1}|,
  \ M_{c,2}=\max |F_{c,2}|,
 \]
 \[
 |F_i({\bf x},{\bf s},\mbox{\boldmath $\psi$})-F_i(\tilde{\bf x},\tilde{\bf s},\tilde{\mbox{\boldmath $\psi$}})|\leq
 \lambda_i\left[\sum_{k=1}^{n}|x_{k}^{}-\tilde x_{k}^{}|
 +\sum_{k=1}^{m}|s_{k}^{}-\tilde s_{k}^{}|+\sum_{k=1}^{n}|\psi_{k}^{}-\tilde \psi_{k}^{}|\right],
 \ i=1,...n,
 \]
 \[
 |F_{s,j}({\bf x},{\bf s},c_{t_0}^{})-
 F_{s,j}(\tilde{\bf x},\tilde{\bf s},\tilde c_{t_0}^{})|
 \leq
 \lambda_{s,j}\left[\sum_{k=1}^{n}|x_{k}^{}-\tilde x_{k}^{}|
 +\sum_{k=1}^{m}|s_{k}^{}-\tilde s_{k}^{}|+|c_{t_0}^{}-\tilde c_{t_0}^{}|
 \right],
 \ j=1,...m,
 \]
 \[
 |F_{\psi,i}(\mbox{\boldmath $\psi$},{\bf s},c_{t_0}^{})-
 F_{\psi,i}(\tilde{\mbox{\boldmath $\psi$}},\tilde{\bf s},\tilde c_{t_0}^{})|
 \leq
 \lambda_{\psi,i}\left[\sum_{k=1}^{n}|\psi_{k}^{}-\tilde \psi_{k}^{}|
 +\sum_{k=1}^{m}|s_{k}^{}-\tilde s_{k}^{}|+|c_{t_0}^{}-\tilde c_{t_0}^{}|
 \right],
 \ i=1,...,n,
 \]
 \[
 |F_{L}({\bf x},{\bf s},\mbox{\boldmath $\psi$},c_{t_0}^{})-
 F_{L}(\tilde{\bf x},\tilde{\bf s},\tilde{\mbox{\boldmath $\psi$}},\tilde c_{t_0}^{})|\leq
 \lambda_{L}\left[\sum_{k=1}^{n}|x_{k}^{}-\tilde x_{k}^{}|
 +\sum_{k=1}^{m}|s_{k}^{}-\tilde s_{k}^{}|+\sum_{k=1}^{n}|\psi_{k}^{}-\tilde \psi_{k}^{}|+|c_{t_0}^{}-\tilde c_{t_0}^{}|
 \right],
 \]
 \[
 |F_{c,1}({\bf x},{\bf s},\mbox{\boldmath $\psi$},c_{t_0}^{})-F_{c,1}(\tilde{\bf x},\tilde{\bf s},\tilde{\mbox{\boldmath $\psi$}},\tilde c_{t_0}^{})|\leq \lambda_{c,1}\left[\sum_{k=1}^{n}|x_{k}^{}-\tilde x_{k}^{}|
 +\sum_{k=1}^{m}|s_{k}^{}-\tilde s_{k}^{}|+\sum_{k=1}^{n}|\psi_{k}^{}-\tilde \psi_{k}^{}|+|c_{t_0}^{}-\tilde c_{t_0}^{}|
 \right],
 \]
 \[
 |F_{c,2}({\bf x},{\bf s},\mbox{\boldmath $\psi$},c_{t_0}^{})-F_{c,2}(\tilde{\bf x},\tilde{\bf s},\tilde{\mbox{\boldmath $\psi$}},\tilde c_{t_0}^{})|\leq
 \lambda_{c,2}\left[\sum_{k=1}^{n}|x_{k}^{}-\tilde x_{k}^{}|
 +\sum_{k=1}^{m}|s_{k}^{}-\tilde s_{k}^{}|+\sum_{k=1}^{n}|\psi_{k}^{}-\tilde \psi_{k}^{}|+|c_{t_0}^{}-\tilde c_{t_0}^{}|
 \right],
 \]
 {\it when}
 $(t_0^{},t)\in[0,\  T_1]\times[0,\  T_1]$ and the functions
 $x_{i}^{}$, $s_{j}^{}$, $\psi_{i}^{}$, $L$, $c$, $c_{t_0}^{}$ satisfy the assumptions (a)-(c).

 Then, integral system \eqref{5.4}-\eqref{5.9}
 has a unique solution $x_{i}^{}$, $s_{j}^{}$, $\psi_{i}^{}$, $L$, $c$, $c_{t_0}^{}$,
 $\in C^0([0,\ T]\times[0,\  T])$,
 where
 \[
 T=\min\left\{T_1,\frac{h_{x,1}}{M_1},...,\frac{h_{x,n}}{M_n},
 \sqrt{\frac{h_{s,1}}{M_{s,1}}},...,\sqrt{\frac{h_{s,m}}{M_{s,m}}},
 \sqrt{\frac{h_{\psi,1}}{M_{\psi,1}}},...,\sqrt{\frac{h_{\psi,n}}{M_{\psi,n}}},
 \sqrt{\frac{h_{L}}{M_{L}}}, \sqrt{\frac{h_{c,1}}{2M_{c,1}}},
 \frac{h_{c,2}}{M_{c,2}}
 \right\}.
 \]
 Moreover, $T$ satisfies the following condition,
 \begin{equation}                                        \label{5.17}
  aT^2+bT<1,
 \end{equation}
 where
 \begin{equation}                                        \label{5.18}
  a=\sum_{j=1}^{m}\lambda_{s,j}+\sum_{i=1}^{n}\lambda_{\psi,i}+\lambda_{L}+2\lambda_{c,1},\
  b=\sum_{i=1}^{n}\lambda_{i}+\lambda_{c,2}.
 \end{equation}
 \end{thm}

 \begin{proof}
 Denote by $\Omega$ the space of continuous functions
 $x_{i}(t_0^{},t)$, $s_{j}(t_0^{},t)$, $\psi_{i}(t_0^{},t)$, $L(t_0^{})$, $c(t_0^{},t)$,
 $c_{t_0}^{}(t_0^{},t)$,\ $t_0^{}\in[0,\ T]$,\ $t\in[0,\ T]$, and introduce
 the norm
 \[
 ||({\bf x},{\bf s},\mbox{\boldmath $\psi$},L,c,c_{t_0}^{})||=\sum_{i=1}^{n}\max_{\Omega}|x_i^{}|
 +\sum_{j=1}^{m}\max_{\Omega}|s_j^{}|+\sum_{i=1}^{n}\max_{\Omega}|\psi_i^{}|
  +\max_{\Omega}|L|+
 \max_{\Omega}|c|+\max_{\Omega}|c_{t_0}^{}|.
 \]
 Consider the map
 $({\bf x}^*,{\bf s}^*,\mbox{\boldmath $\underline{\psi}^*$},L^{*},c^{*},c_{t_0}^{*})=A({\bf x},{\bf s},\mbox{\boldmath $\psi$},L,c,c_{t_0}^{})$,
 where $({\bf x}^*,{\bf s}^*,\mbox{\boldmath $\underline{\psi}^*$},L^{*},c^{*},c_{t_0}^{*})=$
 RHS of equations \eqref{5.4}-\eqref{5.9}.
 Let us prove that $A$ maps $\Omega$ into itself. Indeed,
 \[
 |x_{i}^*-X_{i,0}|\leq M_iT\leq h_{x,i},\ \ i=1,...,n
 \]
 \[
 |s_{j}^*-S_{j}^*|\leq M_{s,j}T^2\leq h_{s,j}, \ \ 
 |\underline{\psi}_{i}^*-\psi_i^*|\leq M_{\psi,i}T^2\leq h_{\psi,i},\ \
 \ i=1,...,n, \ j=1,...,m,
 \]
 \[
  |L^*-\Sigma|\leq M_{L}T^2\leq h_{L},
  \ |c^{*}-\Sigma|\leq 2M_{c,1}T^2 \leq h_{c,1},
  \ |c_{t_0}^{*}-\sigma_a|\leq M_{c,2}T\leq h_{c,2}.
 \]
\noindent
 Consider $(\tilde{\bf x},\tilde{\bf s},\tilde{\mbox{\boldmath $\psi$}}, \tilde L, \tilde c, \tilde c_{t_0}^{})\in\Omega$
 and let
 $(\tilde{\bf x}^*,\tilde{\bf s}^*, \tilde{\mbox{\boldmath $\psi$}}^*, \tilde L^*, \tilde c^*, \tilde c_{t_0}^{*})
 =A(\tilde{\bf x},\tilde{\bf s},\tilde{\mbox{\boldmath $\psi$}},\tilde L, \tilde c, \tilde c_{t_0}^{})$.
 It is possible to obtain
 \[
 |x_{i}^{*}-\tilde x_{i}^{*}|\leq \lambda_i T
 ||({\bf x},{\bf s},\mbox{\boldmath $\psi$},L, c, c_{t_0}^{})-
 (\tilde{\bf x},\tilde{\bf s},\tilde{\mbox{\boldmath $\psi$}}, \tilde L, \tilde c, \tilde c_{t_0}^{})||,\ i=1,...,n,
 \]
 \[
 |s_{j}^{*}-\tilde s_{j}^{*}|\leq
 \lambda_{s,j}T^2||({\bf x},{\bf s},\mbox{\boldmath $\psi$},L, c, c_{t_0}^{})-
 (\tilde{\bf x},\tilde{\bf s},\tilde{\mbox{\boldmath $\psi$}}, \tilde L, \tilde c, \tilde c_{t_0}^{})||,
 \ j=1,...,m,
 \]
  \[
 |\underline{\psi}_{i}^{*}-\tilde \psi_{i}^{*}|\leq
 \lambda_{\psi,i}T^2||({\bf x},{\bf s},\mbox{\boldmath $\psi$},L, c, c_{t_0}^{})-
 (\tilde{\bf x},\tilde{\bf s},\tilde{\mbox{\boldmath $\psi$}}, \tilde L, \tilde c, \tilde c_{t_0}^{})||,
 \ i=1,...,n,
 \]
 \[
 |L^{*}-\tilde L^{*}|\leq
 \lambda_{L}T^2||({\bf x},{\bf s},\mbox{\boldmath $\psi$},L, c, c_{t_0}^{})-
 (\tilde{\bf x},\tilde{\bf s},\tilde{\mbox{\boldmath $\psi$}}, \tilde L, \tilde c, \tilde c_{t_0}^{})||,
 \]
 \[
 |c^{*}-\tilde c^{*}|\leq
 2\lambda_{c,1}T^2||({\bf x},{\bf s},\mbox{\boldmath $\psi$},L, c, c_{t_0}^{})-
(\tilde{\bf x},\tilde{\bf s},\tilde{\mbox{\boldmath $\psi$}}, \tilde L, \tilde c, \tilde c_{t_0}^{})||,
 \]
 \[
 |c_{t_0}^{*}-\tilde c_{t_0}^{*}|\leq
 \lambda_{c,2}T||({\bf x},{\bf s},\mbox{\boldmath $\psi$},L, c, c_{t_0}^{})-
 (\tilde{\bf x},\tilde{\bf s},\tilde{\mbox{\boldmath $\psi$}}, \tilde L, \tilde c, \tilde c_{t_0}^{})||.
 \]
\noindent
 Therefore,
 \[
 ||({\bf x}^{*},{\bf s}^{*}, \mbox{\boldmath $\underline{\psi}^*$}, L^{*}, c^{*}, c_{t_0}^{*})-
 (\tilde{\bf x}^{*},\tilde{\bf s}^{*}, \tilde{\mbox{\boldmath $\psi$}}^*, \tilde L^{*}, \tilde c^{*}, \tilde c_{t_0}^{*})||
 \leq \Lambda
 ||({\bf x},{\bf s},\mbox{\boldmath $\psi$},L, c, c_{t_0}^{})-
 (\tilde{\bf x},\tilde{\bf s},\tilde{\mbox{\boldmath $\psi$}}, \tilde L, \tilde c, \tilde c_{t_0}^{})||,
 \]
 where
 \[
 \Lambda=aT^2+bT.
 \]
 According to \eqref{5.17} $\Lambda<1$, proving Theorem \ref{thm}.
 \end{proof}

%%%%%%%%%%%%%%%%%%%%%%%%%%%%%%%%%%%%%%%%%%%%%%%%%%%%%%%%%%%%%%%%%%%%%%%%%%%
%%%%%%%%%%%%%%%%%%%--------NEW SECTION-----------%%%%%%%%%%%%%%%%%%%%%%%%%%
%%%%%%%%%%%%%%%%%%%%%%%%%%%%%%%%%%%%%%%%%%%%%%%%%%%%%%%%%%%%%%%%%%%%%%%%%%%

 \section{Numerical applications} \label{n6}

Numerical simulations have been performed to test the behavior of the model formulated in Sec. \ref{n4}. Specifically, 
 we have considered the same biofilm system of Sec. \ref{n3}
 composed of 3 microbial species and 3 dissolved substrates. The planktonic species present
 in the bulk liquid are able to initiate the biofilm formation through the attachment process,
 penetrate the biofilm matrix once constituted and establish where the most appropriate 
 growth conditions are found. We have explored two ideal biological situations. In the first case the species 
 $\psi_3^*$ is not initially present in the bulk liquid but it arrives at time $t=t_1$ and starts to 
 attach to the external surface of the biofilm as well as penetrate the biofilm matrix. In the second case,
 the species $\psi_3^*$ is not able to attach to the biofilm surface ($v_{a,3}=0$) but it can establish in 
 sessile form through the colonization process. These biological situations will be referred to as \textit{Case 2} and \textit{Case 3}.

 \noindent
 The reaction terms $r_{M,i}$ and $r_{S,j}$ in equations \eqref{4.1} and \eqref{4.6} have been 
 adopted according to \eqref{3.10} and \eqref{3.11}. The values for $\Psi_i$ on the free boundary have been set according to \eqref{3.6} and \eqref{3.7}. 
 The values for $S_j$ on the  free boundary are reported in Table \ref{t3.1}. 
 The reaction terms concerning the colonization process 
 in equations \eqref{4.1} and \eqref{4.8} are modelled using
 Monod type kinetics and are expressed as

 \begin{equation}                                        \label{6.1}
       r_1 = \frac{k_{col,1}}{\rho} \frac{S_1}{K_1+S_1} \Psi_1, \                                     
			 r_2 = \frac{k_{col,2}}{\rho} \frac{S_2}{K_2+S_2} \Psi_2, \
       r_3 = \frac{k_{col,3}}{\rho} \frac{S_3}{K_3+S_3} \Psi_3,
\end{equation}

\begin{equation}                                        \label{6.2}
       r_{\Psi,1} = - \frac{\rho}{Y_{\Psi,1}} r_1,\
       r_{\Psi,2} = - \frac{\rho}{Y_{\Psi,2}} r_2,\
       r_{\Psi,3} = - \frac{\rho}{Y_{\Psi,3}} r_3
\end{equation}

\noindent
 where $k_{col,1}$, $k_{col,2}$, $k_{col,3}$ are the maximum colonization
 rates of motile species, and $Y_{\Psi,1}$, $Y_{\Psi,2}$, $Y_{\Psi,3}$ are the yields of the
 sessile species on planktonic ones. The values of such kinetic parameters and the diffusion coefficients for 
 $\Psi_i$ are reported in Table \ref{t6.1}. Note that for these ideal biological situations, all the species are supposed to have 
 colonization properties.

 \begin{table}[ht]
 \begin{footnotesize}
 \begin{center}
 \begin{tabular}{llccc}
 \hline
 {\textbf{Parameter}} & {\textbf{Definition}} & {\textbf{Unit}} & {\textbf{Value}}
 \\
 \hline
 $k_{col,1}$     & Maximum colonization rate for $\Psi_1$           &  $d^{-1}$                              & $2.5$           \\
 $k_{col,2}$     & Maximum colonization rate for $\Psi_2$           &  $d^{-1}$                              & $2.5$           \\
 $k_{col,3}$     & Maximum colonization rate for $\Psi_3$           &  $d^{-1}$                              & $2.5$           \\
 $Y_{\Psi,1}$    & Yield of $X_1$ on $\Psi_1$                       &  $--$                                  & $2\cdot10^{-7}$           \\
 $Y_{\Psi,2}$    & Yield of $X_2$ on $\Psi_2$                       &  $--$                                  & $2\cdot10^{-7}$           \\
 $Y_{\Psi,3}$    & Yield of $X_3$ on $\Psi_3$                       &  $--$                                  & $2\cdot10^{-7}$           \\
 $D_{\Psi,1}$    & Diffusion coefficient of $\Psi_1$ in biofilm     &  $m^2 \ d^{-1}$                        & $10^{-5}$       \\
 $D_{\Psi,2}$    & Diffusion coefficient of $\Psi_2$ in biofilm     &  $m^2 \ d^{-1}$                        & $10^{-5}$       \\
 $D_{\Psi,3}$    & Diffusion coefficient of $\Psi_3$ in biofilm     &  $m^2 \ d^{-1}$                        & $10^{-5}$       \\
\hline 
\end{tabular}\\
 \caption{Invasion parameters used for model simulations} \label{t6.1}
 \end{center}
 \end{footnotesize}
 \end{table}

\noindent
Numerical simulations have been performed for \textit{Case 2} and \textit{Case 3} by considering 
a final simulation time $T=10 \ d$. The results are summarized in Figs. \ref{f6.1}-\ref{f6.2} for \textit{Case 2}, and
in Figs. \ref{f6.3}-\ref{f6.4} for \textit{Case 3}.
\\

\textit{Case 2: Attachment and colonization of microbial species $\psi_3^*$}

 \begin{figure}
   \fbox{\includegraphics[width=0.7\textwidth, keepaspectratio]{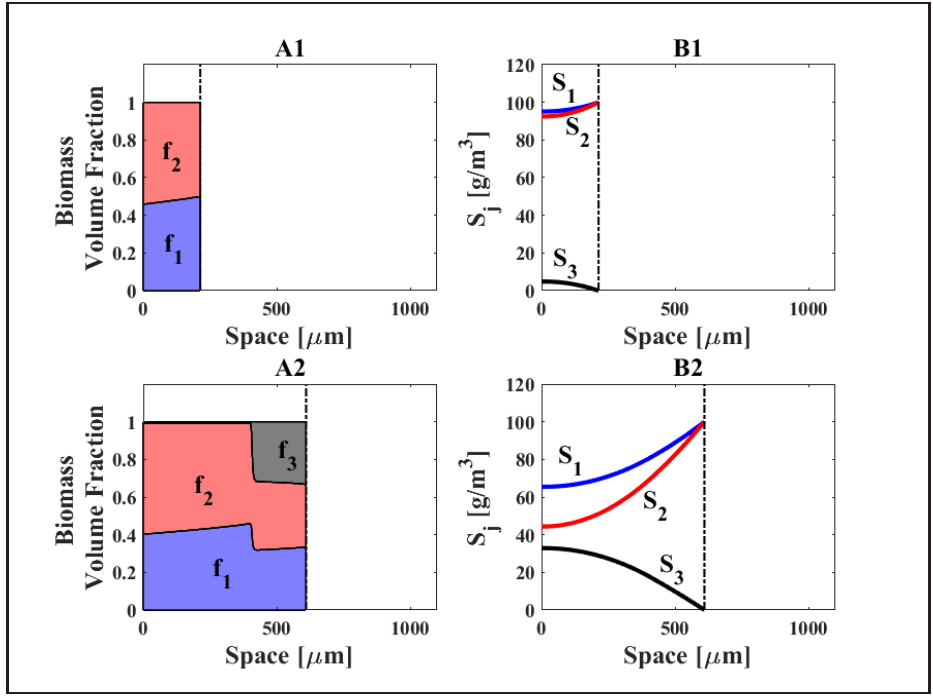}}   
 \caption{Biofilm composition (A1-A2) and substrate distribution (B1-B2) for \textit{Case 2}, under attachment regime, at time
 $t=0.25 \ d$ (top) and $t=0.50 \ d$ (bottom).}  \label{f6.1}
 \end{figure}

 \begin{figure}
    \fbox{\includegraphics[width=0.7\textwidth, keepaspectratio]{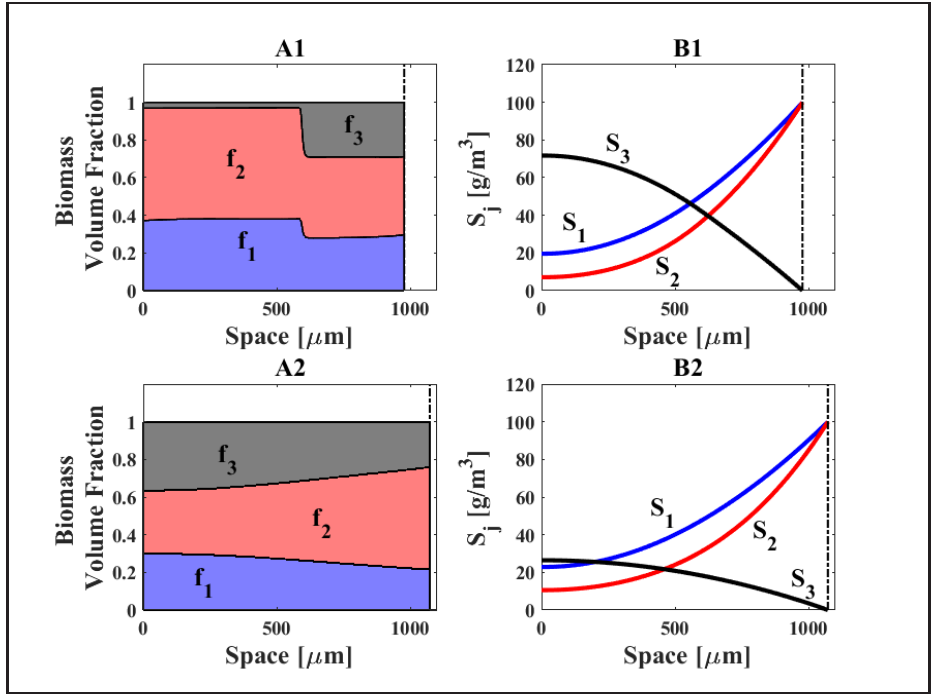}}   
 \caption{Biofilm composition (A1-A2) and substrate distribution (B1-B2) for \textit{Case 2}, under detachment regime, at time
 $t=1 \ d$ (top) and $t=10 \ d$ (bottom).}  \label{f6.2}
 \end{figure}

The results reported in Figs. \ref{f6.1} and \ref{f6.2} highlight model capability to 
reproduce both the attachment and colonization phenomena that strongly affect biofilm lifecycle. 
In particular, it is possible to notice that during the initial phase of biofilm formation, 
the biofilm undergoes the same development illustrated in Sec. \ref{n3} and reported 
in Fig. \ref{f6.2}. However, at time $t=0.50 \ d$ it is visible that the volume 
fraction of the third species $f_3$ is slightly positive even in the region $z< c(t_1,t)$ due to
the colonization phenomenon (Fig. \ref{f6.1}(A2)). Going on with the simulation time, $f_3$
increases all over the biofilm leading to a higher biofilm thickness at time $t=1 \ d$ (Fig. \ref{f6.2}(A1)). At the final
simulation time $t=10 \ d$ and under detachment regime, the biofilm is constituted by all the species inhabiting the bulk liquid (Fig. \ref{f6.2}(A2))
conversely to the numerical results reported in Sec. \ref{n3} where the complete washout of species $f_3$ has been observed.
The different biofilm stratification affects substrate trends as it is possible to notice that 
at the final simulation time, $S_3$ concentration is much lower when compared to the numerical example of pure attachment regime. 
\\ 

\textit{Case 3: Pure colonization of microbial species $\psi_3^*$}

 \begin{figure}
     \fbox{\includegraphics[width=0.7\textwidth, keepaspectratio]{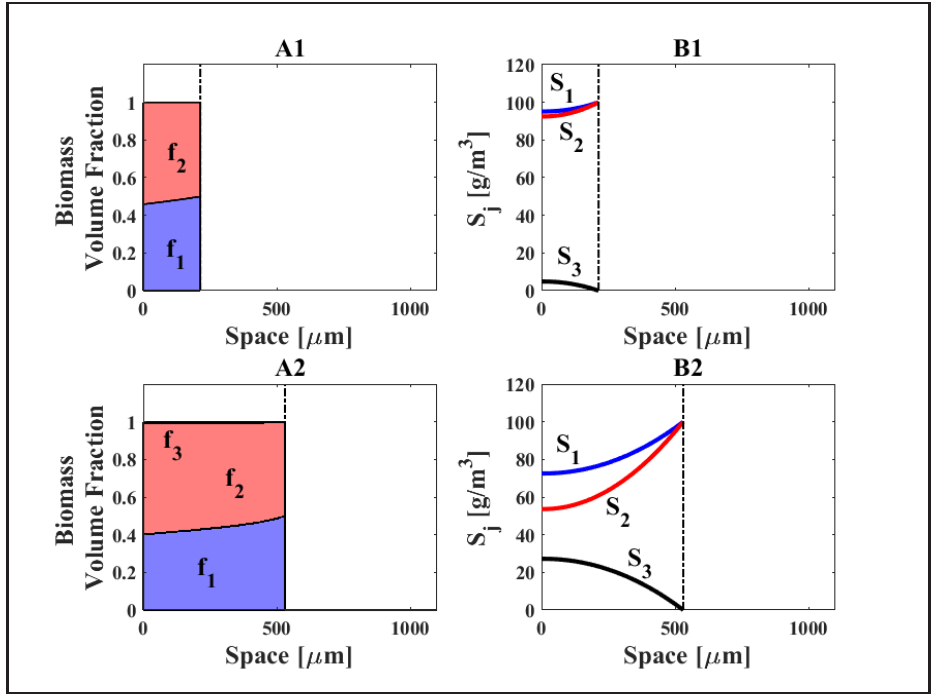}}   
 \caption{Biofilm composition (A1-A2) and substrate distribution (B1-B2) for \textit{Case 3}, under attachment regime, at time
 $t=0.25 \ d$ (top) and $t=0.50 \ d$ (bottom).}  \label{f6.3}
 \end{figure}

 \begin{figure}
     \fbox{\includegraphics[width=0.7\textwidth, keepaspectratio]{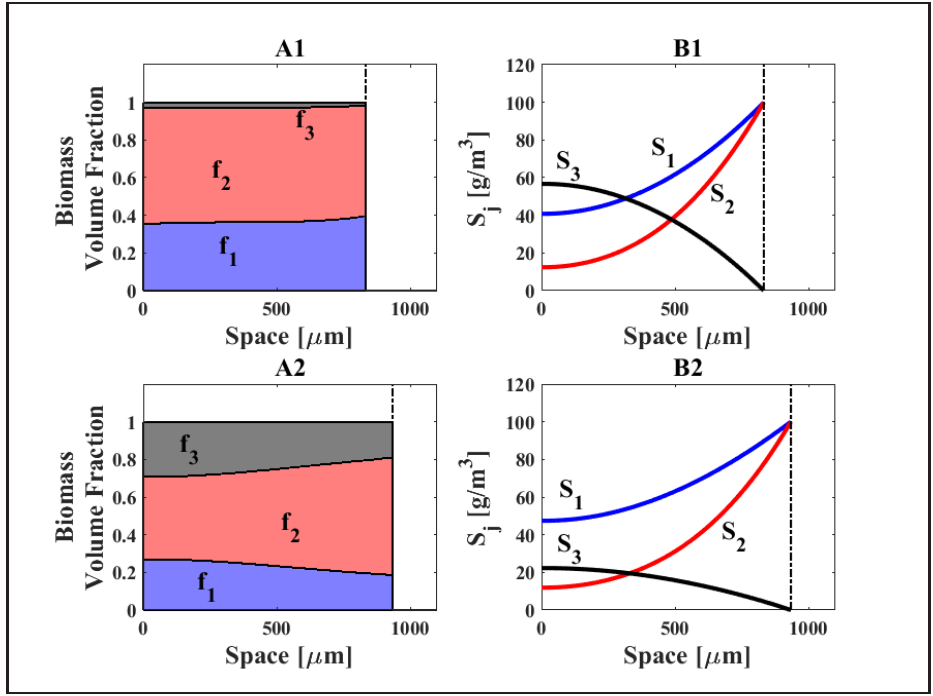}}   
 \caption{Biofilm composition (A1-A2) and substrate distribution (B1-B2) for \textit{Case 3}, under detachment regime, at time
 $t=1 \ d$ (top) and $t=10 \ d$ (bottom).}  \label{f6.4}
 \end{figure}

Figs. \ref{f6.3} and \ref{f6.4} illustrate the biofilm development and substrate trends 
when the species $\psi_3^*$ is not able to attach to the biofilm surface, but it can penetrate the biofilm
matrix and establish in sessile form. According to Figs. \ref{f6.3}-\ref{f6.4}, numerical results reveal that
for all simulation times the biofilm thickness is smaller compared to the pure attachment case.
This contributes to have different substrate trends within the biofilm (Figs. \ref{f6.3}-\ref{f6.4}(B1-B2)). 
In terms of biomass distribution, 
it is possible to notice that the third species grow in sessile form in the inner layers of the biofilm 
where there is the highest $S_3$ concentration. The biomass stratification and subtrate trends at the final simulation time
resemble the one achieved for \textit{Case 2}. Such results highlight an important feature of the model:
the attachment and colonization phenomena are both dependent on the planktonic cells present in the 
bulk liquid. They can occur simultaneously reproducing the case of planktonic cells able to attach to the surface 
and penetrate the biofilm matrix. Conversely, the planktonic cells can be characterized by a certain motility 
which drives them to the biofilm region where there are the most appropriate conditions for their growth.

%%%%%%%%%%%%%%%%%%%%%%%%%%%%%%%%%%%%%%%%%%%%%%%%%%%%%%%%%%%%%%%%%%%%%%%%%%%
%%%%%%%%%%%%%%%%%%%--------NEW SECTION-----------%%%%%%%%%%%%%%%%%%%%%%%%%%
%%%%%%%%%%%%%%%%%%%%%%%%%%%%%%%%%%%%%%%%%%%%%%%%%%%%%%%%%%%%%%%%%%%%%%%%%%%

 \section{Conclusion} \label{n7}

The proposed model comprehensively describes the transition from planktonic to sessile phenotype which governs 
the biofilm dynamics. This allows to properly reproduce the evolution of biofilms starting from the initial 
formation and including the establishment and growth of new species. The criticism of Wanner and Gujer type models, 
discussed in \cite{klapper2011exclusion}, is here emphasized through a numerical example. Such models are not able 
to properly describe the growth of microbial species which do not participate in the initial biofilm formation
attaching later to a pre-existing aggregate. The presented model is able to overcome this issue as 
it considers both the initial attachment phase and the growth of new sessile species within the biofilm mediated by the invasion process. 
The modelling of the initial phase of biofilm formation allows to describe the biofilm growth without arbitrarily fixing 
the initial composition of the biofilm. The existence and uniqueness of solutions is proved in the case of attachment regime. 
Numerical examples are provided to show model capability to reproduce the different stages of biofilm growth 
as affected by the planktonic phenotype. Future work may be related to the 
role of biofilm porosity on planktonic species diffusion and the qualitative analysis under detachment regime.

%%%%%%%%%%%%%%%%%%%%%%%%%%%%%%%%%%%%%%%%%%%%%%%%%%%%%%%%%%%%%%%%%%%%%%%%%%%
%%%%%%%%%%%%%%%%%%%--------NEW SECTION-----------%%%%%%%%%%%%%%%%%%%%%%%%%%
%%%%%%%%%%%%%%%%%%%%%%%%%%%%%%%%%%%%%%%%%%%%%%%%%%%%%%%%%%%%%%%%%%%%%%%%%%%

 \section*{Acknowledgements}

This study has been performed under the auspices of the G.N.F.M. of Indam.
 The authors acknowledge the Progetto Giovani G.N.F.M. 2019 \textit{Modellazione ed analisi di sistemi microbici complessi: applicazione ai biofilm},
 the project \textit{VOLAC - Valorization of OLive oil wastes for sustainable
 production of biocide-free Antibiofilm Compounds} of Cariplo Foundation (grant number 2017-0977)
 and the program "Programma Operativo Nazionale Ricerca e Innovazione (PON RI 2014/2020) Action I.1 - Innovative PhDs with industrial characterization"
 for financial support.

% ------------------------------------------------------------------------
\end{document}